\documentclass[11pt]{article}

\usepackage[utf8]{inputenc}
\usepackage[T1]{fontenc}
\usepackage{amsmath,amssymb,amsthm,mathtools,booktabs, longtable, tikz}
\usepackage{mathrsfs}
\usepackage[margin=1in,top=1.2in,bottom=1.2in]{geometry}
\usepackage{enumitem}
\usepackage[hidelinks,colorlinks=true,linkcolor=blue,citecolor=blue]{hyperref}
\usepackage{array}
\theoremstyle{plain}
\newtheorem{theorem}{Theorem}[section]
\newtheorem{lemma}[theorem]{Lemma}
\newtheorem{proposition}[theorem]{Proposition}
\newtheorem{corollary}[theorem]{Corollary}
\newtheorem{observation}[theorem]{Observation}
\theoremstyle{definition}
\newtheorem{definition}[theorem]{Definition}

\theoremstyle{remark}
\newtheorem{remark}[theorem]{Remark}

\newcommand{\R}{\mathbb{R}}
\newcommand{\ZZ}{\mathbb{Z}}
\DeclareMathOperator{\wt}{wt}
\newcommand{\colA}{\mathrm{A}}
\newcommand{\colB}{\mathrm{B}}
\newcommand{\colC}{\mathrm{C}}
\newcommand{\colD}{\mathrm{D}}

\title{\textbf{Exact Dominion of the Prism Graph:\\
Enumeration by Congruence Class via Cyclic Words}}

\author{Julian~Allagan\thanks{Department of Mathematics, Computer Science, and Engineering Technology,
Elizabeth City State University, Elizabeth City, NC 27909, USA. 
Email: \texttt{adallagan@ecsu.edu}}
}

\date{}

\begin{document}

\maketitle
\begin{abstract}
Let $G_n=C_n\square P_2$ denote the prism (circular ladder) graph on $2n$ vertices.
By encoding column configurations as cyclic words, domination is reduced to local
Boolean constraints on adjacent factors.
This framework yields explicit formulas for the dominion $\zeta(G_n)$, stratified by
$n \bmod 4$, with the exceptional cases $n\in\{3,6\}$ confirmed computationally.
Together with the known domination numbers $\gamma(G_n)$, these results expose
distinct arithmetic regimes governing optimal domination, ranging from rigid forcing
to substantial enumerative flexibility, and motivate quantitative parameters for
assessing structural robustness in parametric graph families.
\end{abstract}

\vspace{0.3cm}

\vspace{0.3cm}

\section{Introduction}\label{sec:intro}

The prism (or circular ladder) graph $G_n = C_n \square P_2$
is a basic Cartesian product that combines cyclic symmetry with a fixed two-layer structure.
It arises naturally as a Cayley graph of the dihedral group and has long served as a
canonical test case in domination theory, where global regularity coexists with strong
local constraints.

A set $S\subseteq V(G)$ is a \emph{dominating set} if every vertex lies in its closed
neighborhood $N[S]$.  The domination number $\gamma(G)$ records the minimum size of such a
set but does not distinguish between rigid and flexible optimal domination.
This distinction is captured by the \emph{dominion}, introduced in~\cite{allagan-bobga-2021},
$
\zeta(G)=
\bigl|\{\,S\subseteq V(G): |S|=\gamma(G)\text{ and }S\text{ dominates }G\,\}\bigr|,
$
which counts the number of minimum dominating configurations.

Exact dominion formulas are known for only a handful of graph families, including paths
and cycles~\cite{allagan-bobga-2021}, with recent extensions to grids~\cite{su-allagan-2024}.
More generally, upper bounds on $\zeta(G)$ in terms of $\gamma(G)$
(e.g.,~\cite{godbole-jamieson-jamieson-2014,connolly-etal-2016}) show that exponential growth
is unavoidable in broad classes, even though concrete families often exhibit far stronger
structural forcing.  For instance, Petr, Portier, and Versteegen~\cite{petr-portier-versteegen-2024}
proved that forests satisfy $\zeta(F)\le 5^{\gamma(F)}$, illustrating how rigid local structure
can persist beneath coarse global bounds.

For prism graphs, the domination number is completely understood.
Grinstead and Slater~\cite{grinstead-slater-1991} established the following congruence-class
formula.

\begin{theorem}[{\cite{grinstead-slater-1991}}]\label{thm:gamma-prism}
For $n\ge 3$,
\[
\gamma(G_n)=
\begin{cases}
\dfrac{n}{2}, & n\equiv 0\pmod{4},\\[6pt]
\dfrac{n+1}{2}, & n\equiv 1,3\pmod{4},\\[6pt]
\dfrac{n}{2}+1, & n\equiv 2\pmod{4}.
\end{cases}
\]
\end{theorem}

By contrast, explicit formulas for the dominion $\zeta(G_n)$ do not appear to have been
previously determined.
The main result of this paper is a complete determination of $\zeta(G_n)$ for all $n\ge 3$.
The formulas fall into three qualitatively distinct regimes—constant, linear, and quadratic—
according to $n\bmod 4$, with exceptional behavior at $n=3$ and $n=6$.
These regimes reflect local forcing mechanisms that are invisible at the level of
$\gamma(G_n)$ alone.

Our method is structural and combinatorial.
We encode dominating sets as cyclic words over a four-letter alphabet representing column
configurations, translating domination into local Boolean constraints on adjacent factors.
Minimum dominating sets correspond to minimum-weight words satisfying these constraints,
reducing the enumeration problem to a rigid language-theoretic classification.
In the most delicate case $n\equiv 2\pmod{4}$, the quadratic growth
$\zeta(G_n)=n(n+2)$ is obtained via an explicit orbit--anchor decomposition and a
constructive bijection.

The paper is organized as follows.
Section~\ref{sec:prelim} introduces notation and the word-based encoding.
Section~\ref{sec:main} proves the exact dominion formulas.
Section~\ref{sec:flexibility} analyzes flexibility parameters suggested by the enumeration.
Section~\ref{sec:future} outlines further directions, including extensions to other
Cartesian products.
Computational verification for small values of $n$ is provided in
Appendix~\ref{app:computation}.

\section{Preliminaries}\label{sec:prelim}
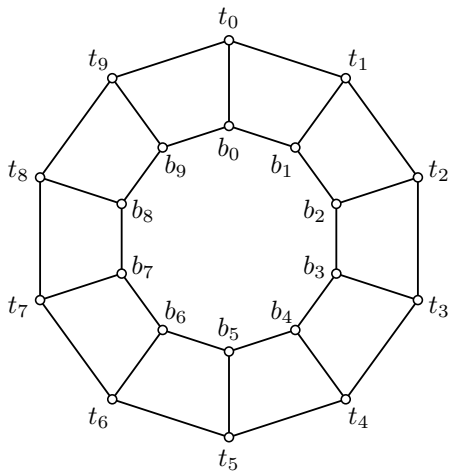
\begin{figure}[t]
\centering
\begin{tikzpicture}[scale=1.2, every node/.style={font=\small}]
  \def\n{10}
  \def\R{2.2}     
  \def\r{1.25}    
  \def\start{90}  
  \def\labsep{0.25} 

  \tikzset{
    vtx/.style={circle, draw, fill=white, inner sep=1.2pt, line width=0.6pt},
    edg/.style={line width=0.7pt},
  }

  \foreach \i in {0,...,9} {
    \coordinate (t\i) at ({\start-360/\n*\i}:\R);
    \coordinate (b\i) at ({\start-360/\n*\i}:\r);
  }

  \foreach \i in {0,...,9} {
    \pgfmathtruncatemacro{\ip}{mod(\i+1,\n)}
    \draw[edg] (t\i) -- (t\ip);
    \draw[edg] (b\i) -- (b\ip);
    \draw[edg] (t\i) -- (b\i);
  }

  \foreach \i in {0,...,9} {
    \node[vtx] at (t\i) {};
    \node[vtx] at (b\i) {};
  }

  \foreach \i in {0,...,9} {
    \pgfmathsetmacro{\ang}{\start-360/\n*\i}

    \node at ({\ang}:{\R+\labsep}) {$t_{\i}$};

    \node at ({\ang}:{\r-\labsep}) {$b_{\i}$};
  }

\end{tikzpicture}
\caption{The prism graph $G_{10}=C_{10}\square P_{2}$ with vertex labels $t_i,b_i$.}
\label{fig:prism}
\end{figure}

Throughout the paper, indices are taken modulo $n$ and we write $\ZZ_n=\{0,1,\dots,n-1\}$.
The prism graph of order $n\ge 3$ is the Cartesian product
$G_n \;=\; C_n \square P_2$,
with vertex set $V(G_n)=\{(t,i),(b,i): i\in\ZZ_n\}$; see Figure~\ref{fig:prism} for an illustration with $n=10$. Edges join $(t,i)$ to $(t,i\pm1)$ and to $(b,i)$, and symmetrically $(b,i)$ to $(b,i\pm1)$
and to $(t,i)$.  Thus $G_n$ consists of two $n$-cycles linked by $n$ rungs, and has $2n$
vertices and $3n$ edges.

A set $S\subseteq V(G_n)$ is a \emph{dominating set} if every vertex lies in the closed
neighborhood $N[S]$.  The \emph{domination number} $\gamma(G_n)$ is the minimum size of such a
set, and the \emph{dominion} $\zeta(G_n)$ counts all dominating sets of this minimum size.
Since the prism graphs considered here are labeled, distinct subsets of $V(G_n)$ are counted
separately, even if they are related by cyclic rotation.

To study domination combinatorially, we encode dominating sets as cyclic words.
A set $S\subseteq V(G_n)$ corresponds uniquely to a word
\[
w = (w_0,w_1,\dots,w_{n-1}) \in \{00,10,01,11\}^n,
\]
where the two bits of $w_i$ record membership of $(t,i)$ and $(b,i)$ in $S$.
For convenience we use the alphabet
\[
\colC=00,\qquad \colA=10,\qquad \colB=01,\qquad \colD=11.
\]
For a letter $x\in\{\colC,\colA,\colB,\colD\}$ we denote its top and bottom bits by
$x_t,x_b\in\{0,1\}$, so that $(w_i)_t,(w_i)_b$ represent the selections in column $i$.

The \emph{weight} of a word $w$ is defined by
\[
\wt(w)\coloneqq 
\#\{i:w_i=\colA\}+\#\{i:w_i=\colB\}+2\,\#\{i:w_i=\colD\},
\]
so that $\wt(w)=|S|$.  A word is called a \emph{minimum dominating word} if it encodes a
dominating set and satisfies $\wt(w)=\gamma(G_n)$.

Domination in $G_n$ translates directly into local constraints on adjacent letters of $w$.

\begin{lemma}[Local domination constraints]\label{lem:local-dom}
Let $w\in\{\colC,\colA,\colB,\colD\}^n$ encode a set $S\subseteq V(G_n)$.
Then $S$ dominates $G_n$ if and only if, for every $i\in\ZZ_n$, both
\begin{align}
\label{eq:dom-top}
(w_i)_t \,\vee\, (w_i)_b \,\vee\, (w_{i-1})_t \,\vee\, (w_{i+1})_t &= 1,\\
\label{eq:dom-bot}
(w_i)_b \,\vee\, (w_i)_t \,\vee\, (w_{i-1})_b \,\vee\, (w_{i+1})_b &= 1
\end{align}
hold.
\end{lemma}

\begin{proof}
Fix $i\in\ZZ_n$.  The vertex $(t,i)$ is dominated if and only if at least one of the following
occurs: $(t,i)\in S$, $(b,i)\in S$, $(t,i-1)\in S$, or $(t,i+1)\in S$.  Translating these four
possibilities into the word representation yields \eqref{eq:dom-top}.
The same argument applied to $(b,i)$, whose neighbors are $(t,i)$, $(b,i-1)$, and $(b,i+1)$,
gives \eqref{eq:dom-bot}.
\end{proof}

This encoding reduces domination in $G_n$ to a finite collection of local Boolean constraints.
In the sections that follow, these constraints will be exploited to derive strong structural
restrictions on minimum dominating words, ultimately leading to exact enumeration of the
dominion $\zeta(G_n)$.

\section{Main results: exact dominion formulas}\label{sec:main}

\begin{theorem}[Exact dominion of the prism]\label{thm:main-dominion}
For $n\ge 3$ and $G_n=C_n\square P_2$,
\[
\zeta(G_n)=
\begin{cases}
9, & n=3,\\
4, & n\equiv 0\pmod{4},\\
2n, & n\equiv 1,3\pmod{4},\\
51, & n=6,\\
n(n+2), & n\equiv 2\pmod{4},\ n\ge 10.
\end{cases}
\]
\end{theorem}

\medskip
We now turn the local domination constraints from Section~\ref{sec:prelim} into global structure.
Throughout, we use the word encoding from Section~\ref{sec:prelim}, where a dominating set
$S\subseteq V(G_n)$ corresponds to a word
$w=(w_0,\dots,w_{n-1})\in\{\colA,\colB,\colC,\colD\}^n$ indexed by $\ZZ_n$.
By Lemma~\ref{lem:local-dom}, domination is equivalent to the two row-wise constraints at every
index.  A \emph{minimum dominating word} satisfies $\wt(w)=\gamma(G_n)$.

All arguments below are structural and apply uniformly for $n\ge4$.
The two remaining small cases $n\in\{3,6\}$ are handled separately by direct enumeration
(Appendix~\ref{app:computation}).
\subsection{Forcing and exclusion}\label{subsec:forcing}

Before splitting into congruence classes, we isolate three forcing rules that repeatedly
compress the search space.  The first shows that consecutive empty columns are incompatible
with minimality, and the next two convert local emptiness into deterministic neighbor patterns.

\begin{lemma}[$\colC\colC$ forces doubles]\label{lem:CC-forces-D}
Let $w$ be a dominating word. If $w_i=w_{i+1}=\colC$, then $w_{i-1}=w_{i+2}=\colD$.
In particular, at least one of $w_{i-1},w_{i+2}$ equals $\colD$.
\end{lemma}

\begin{proof}
With $w_i=w_{i+1}=\colC$, all four bits in columns $i,i+1$ are $0$.
Apply the top-row constraint (Lemma~\ref{lem:local-dom}) at index $i$:
\[
(w_{i-1})_t \vee (w_i)_t \vee (w_i)_b \vee (w_{i+1})_t
\;=\; (w_{i-1})_t \vee 0 \vee 0 \vee 0 \;=\;1,
\]
so $(w_{i-1})_t=1$. The bottom-row constraint at $i$ similarly forces $(w_{i-1})_b=1$,
hence $w_{i-1}=\colD$. Repeating the same argument at index $i+1$ yields $w_{i+2}=\colD$.
\end{proof}

\begin{lemma}[An empty column forces complementary singleton neighbors]\label{lem:C-forces}
If $w_i=\colC$, then
\[
(w_{i-1})_t\vee(w_{i+1})_t=1
\qquad\text{and}\qquad
(w_{i-1})_b\vee(w_{i+1})_b=1.
\]
In particular, if neither $w_{i-1}$ nor $w_{i+1}$ equals $\colD$, then
$\{w_{i-1},w_{i+1}\}=\{\colA,\colB\}$.
\end{lemma}

\begin{proof}
Substitute $(w_i)_t=(w_i)_b=0$ into the two domination constraints at index $i$.
If neither neighbor is $\colD$, then each neighbor contributes exactly one $1$-bit,
so the two displayed disjunctions force one neighbor to provide the top bit and
the other the bottom bit. Equivalently, the neighbors must be $\colA$ and $\colB$.
\end{proof}

\begin{lemma}[No $\colD$ in minimum words]\label{lem:no-D}
If $n\ge4$ and $w$ is a minimum dominating word for $G_n$, then $w$ contains no $\colD$.
\end{lemma}

\begin{proof}
Assume $w_i=\colD$ for some $i$.  We construct a dominating word $w'$ with
$\wt(w')=\wt(w)-1$, contradicting minimality.

Let $a=(w_{i-1})_t\vee (w_{i+1})_t$ and $b=(w_{i-1})_b\vee (w_{i+1})_b$.
Define
\[
w'_i :=
\begin{cases}
\colA, & a=0,\\
\colB, & a=1 \text{ and } b=0,\\
\colA, & a=1 \text{ and } b=1,
\end{cases}
\qquad\text{and}\qquad
w'_j:=w_j\ \ (j\neq i).
\]
Thus $w'_i$ is a singleton and $\wt(w')=\wt(w)-1$.

Only the constraints at indices $i-1,i,i+1$ can be affected.  We verify them.

\smallskip\noindent
\emph{Index $i$.}
Since $w'_i$ is a singleton, $(w'_i)_t\vee (w'_i)_b=1$, so both
\eqref{eq:dom-top} and \eqref{eq:dom-bot} at $i$ hold.

\smallskip\noindent
\emph{Index $i-1$.}
The only changed literal in \eqref{eq:dom-top} at $i-1$ is $(w_i)_t$, replaced by $(w'_i)_t$.
If $(w_{i-1})_t=1$ or $(w_{i-2})_t=1$, the constraint is already satisfied.
Otherwise $(w_{i-1})_t=(w_{i-2})_t=0$, so the top constraint at $i-1$ reduces to
\[
(w_{i-1})_b \ \vee\ (w'_i)_t \;=\;1.
\]
If $(w_{i-1})_b=1$ we are done; if $(w_{i-1})_b=0$, then the bottom constraint at $i$
(with $(w_i)_b=1$ in $w$) forces $b=1$, and by construction (when $b=1$ and the top row at $i-1$
still needs help) we ensure $(w'_i)_t=1$.  Hence \eqref{eq:dom-top} holds at $i-1$.
The bottom constraint \eqref{eq:dom-bot} at $i-1$ is analogous: if it is not already satisfied by
$(w_{i-1})_t$ or $(w_{i-2})_b$, it reduces to
\[
(w_{i-1})_t \ \vee\ (w'_i)_b \;=\;1,
\]
and the choice of $w'_i$ guarantees $(w'_i)_b=1$ whenever this is needed.

\smallskip\noindent
\emph{Index $i+1$.}
The verification is symmetric to $i-1$ (swap $i-2$ with $i+2$ and reverse roles).
Thus both row constraints remain valid at $i+1$.

Therefore $w'$ is dominating with strictly smaller weight, contradicting the minimality of $w$.
\end{proof}
\subsection{Case $n\equiv 0\pmod4$: complete rigidity}\label{subsec:mod0}

We begin with the most rigid regime.  When $n$ is a multiple of $4$, the local forcing
mechanisms from Section~\ref{subsec:forcing} leave essentially no global freedom:
minimum domination is completely periodic.

\begin{lemma}[Classification for $n\equiv 0\pmod4$]\label{lem:class-mod0}
Let $n=4k\ge4$ and let $w$ be a minimum dominating word for $G_n$.
Then $w\in\{\colC,\colA,\colB\}^n$, with
\[
\#\colC=2k
\qquad\text{and}\qquad
\#\{\colA,\colB\}=2k,
\]
and, up to cyclic rotation,
\[
w\in\bigl\{(\colB\colC\colA\colC)^k,\ (\colA\colC\colB\colC)^k\bigr\}.
\]
\end{lemma}

\begin{proof}
By Theorem~\ref{thm:gamma-prism}, $\wt(w)=\gamma(G_{4k})=2k$.
Lemma~\ref{lem:no-D} excludes $\colD$, so all weight comes from singleton columns.
Consequently, $\#\{\colA,\colB\}=2k$ and the remaining $2k$ columns must be $\colC$.

A factor $\colC\colC$ cannot occur: Lemma~\ref{lem:CC-forces-D} would then force a
$\colD$, again contradicting Lemma~\ref{lem:no-D}.
Since the numbers of $\colC$'s and singletons coincide, the word must therefore
alternate strictly between $\colC$ and a singleton around the entire cycle.

At each $\colC$, the two neighboring singletons are forced to be complementary by
Lemma~\ref{lem:C-forces}. This fixes the singleton pattern up to global row-swap.
Because $n$ is divisible by $4$, the alternation closes consistently, producing
exactly the two displayed cyclic words.
\end{proof}

\begin{corollary}\label{cor:count-mod0}
If $n\equiv 0\pmod4$, then $\zeta(G_n)=4$.
\end{corollary}

\begin{proof}
Lemma~\ref{lem:class-mod0} yields two distinct cyclic patterns.
Each admits exactly two labeled realizations, obtained by shifting the word by one
column, and the global row-swap $\colA\leftrightarrow\colB$ produces no new patterns
beyond these. Hence there are precisely four minimum dominating sets.
\end{proof}

\subsection{Case $n\equiv 1,3\pmod4$: the linear regime}\label{subsec:odd}

We next turn to odd values of $n$, where the local forcing mechanisms remain strong
but no longer close periodically.  Instead, a single local defect propagates around
the cycle, producing a linear family of minimum dominating sets.

\begin{lemma}[Classification for odd $n$]\label{lem:class-odd}
Let $n=2m+1\ge5$ be odd and let $w$ be a minimum dominating word for $G_n$.
Then $w\in\{\colC,\colA,\colB\}^n$ with
\[
\#\colC=m
\qquad\text{and}\qquad
\#\{\colA,\colB\}=m+1,
\]
and there exists a unique index $i$ such that
\[
(w_{i-1},w_i,w_{i+1},w_{i+2})
=(\colC,\colA,\colA,\colC)
\quad\text{or}\quad
(\colC,\colB,\colB,\colC).
\]
Outside this unique block, $\colC$ and singleton letters alternate, and every
$\colC$ is flanked by complementary singletons.
\end{lemma}

\begin{proof}
By Theorem~\ref{thm:gamma-prism}, $\wt(w)=\gamma(G_{2m+1})=m+1$.
Lemma~\ref{lem:no-D} excludes $\colD$, so $w$ consists only of $\colC$, $\colA$, and $\colB$.
Thus $\#\colC=(2m+1)-(m+1)=m$ and $\#\{\colA,\colB\}=m+1$.

As in the even case, $\colC\colC$ cannot occur: Lemmas~\ref{lem:CC-forces-D}
and~\ref{lem:no-D} together forbid adjacent empty columns.
List the $m$ occurrences of $\colC$ in cyclic order and let
$g_1,\dots,g_m\ge1$ denote the numbers of consecutive singletons between successive
$\colC$'s. Since there are $m+1$ singletons in total,
\[
\sum_{j=1}^m g_j=m+1,
\qquad\text{so}\qquad
\sum_{j=1}^m (g_j-1)=1.
\]
Hence exactly one gap has size $2$, yielding a unique factor
$\colC x y \colC$ with $x,y\in\{\colA,\colB\}$.

If this pair were mixed, say $(x,y)=(\colA,\colB)$, then Lemma~\ref{lem:C-forces}
would rigidly propagate an $\colA/\colB$ alternation through every size-$1$ gap.
Because the cycle length is odd, such an alternation cannot close consistently:
returning to the starting column would require a second size-$2$ gap, contradicting
the uniqueness established above. Therefore $x=y$, and the unique gap is necessarily
monochromatic.

Once this block is fixed, Lemma~\ref{lem:C-forces} forces complementary neighbors at
every remaining $\colC$, so the rest of the word is determined by strict alternation.
\end{proof}

\begin{corollary}\label{cor:count-odd}
If $n\ge5$ is odd, then $\zeta(G_n)=2n$.
\end{corollary}

\begin{proof}
By Lemma~\ref{lem:class-odd}, a minimum dominating word is uniquely specified by:
(i) the position of the single size-$2$ gap, which may occur at any of the $n$ cyclic
locations, and (ii) its monochromatic type $\colA\colA$ or $\colB\colB$.
Different choices produce distinct labeled dominating sets, yielding $\zeta(G_n)=2n$.
\end{proof}

\subsection{Case $n\equiv 2\pmod4$, $n\ge10$: the quadratic family}\label{subsec:mod2}

We now consider the remaining congruence class, where two local defects coexist and
interact along the cycle.  This interaction produces a quadratic growth of the dominion.

Assume throughout that $n=4t+2$ with $t\ge2$.  By Lemma~\ref{lem:no-D} and
Theorem~\ref{thm:gamma-prism}, every minimum dominating word lies in
$\{\colC,\colA,\colB\}^n$ and has weight $2t+2$.  Consequently, such a word contains
exactly $2t$ letters $\colC$ and $2t+2$ singleton letters.

\begin{lemma}[Two-gap structure]\label{lem:mod2-two-gap}
Every minimum dominating word for $G_{4t+2}$ has exactly two size-$2$ gaps between
consecutive $\colC$'s, and all remaining gaps have size $1$.
\end{lemma}

\begin{proof}
Let $w$ be a minimum dominating word and list its $2t$ occurrences of $\colC$ in cyclic
order.  Let $g_1,\dots,g_{2t}\ge1$ denote the numbers of consecutive singleton letters
between successive $\colC$'s.  Since there are $2t+2$ singletons in total,
\[
\sum_{j=1}^{2t} g_j = 2t+2,
\qquad\text{so}\qquad
\sum_{j=1}^{2t}(g_j-1)=2.
\]
Thus exactly two gaps have size $2$, while all others have size $1$.
\end{proof}

The next observation shows that these two gaps are rigidly constrained.

\begin{lemma}[Monochromatic size-$2$ gaps]\label{lem:mod2-gap-mono}
Let $n=4t+2\ge10$, and let $w\in\{\colC,\colA,\colB\}^n$ be a minimum dominating word.
Then every size-$2$ gap has the form
\[
\colC\,\varepsilon\,\varepsilon\,\colC
\qquad\text{with }\varepsilon\in\{\colA,\colB\}.
\]
\end{lemma}

\begin{proof}
Suppose, for a contradiction, that a mixed gap occurs.  By symmetry, assume
\[
(w_i,w_{i+1},w_{i+2},w_{i+3})=(\colC,\colA,\colB,\colC).
\]
Applying Lemma~\ref{lem:C-forces} at $i$ and at $i+3$ yields
$w_{i-1}=\colB$ and $w_{i+4}=\colA$, so the forced factor
\[
w_{i-1}\cdots w_{i+4}
=\colB\,\colC\,\colA\,\colB\,\colC\,\colA
\tag{$\ast$}\label{eq:mixed-parity-block}
\]
appears.

By Lemma~\ref{lem:mod2-two-gap}, both gaps adjacent to the two $\colC$'s in
\eqref{eq:mixed-parity-block} have size $1$.  Hence this factor lies inside a maximal
region where $\colC$ and singletons strictly alternate.  Within such a region,
Lemma~\ref{lem:C-forces} fixes the singleton letters up to global swap and forces a
period-$2$ alternation.  Fixing $w_{i-1}=\colB$ therefore enforces the pattern
$\cdots,\colB,\colA,\colB,\colA,\cdots$ throughout the region.

The mixed gap $\colC\colA\colB\colC$ disrupts this forced parity by placing two
consecutive singletons of opposite types between the same pair of $\colC$'s.
On a cycle, a single parity defect cannot be absorbed: returning to the original
parity class would require another mixed gap elsewhere.  This contradicts
Lemma~\ref{lem:mod2-two-gap}, which allows exactly two size-$2$ gaps in total.
Thus every size-$2$ gap must be monochromatic.
\end{proof}

The region separating the two gaps is therefore completely rigid.

\begin{definition}[Backbone]\label{def:backbone}
Set $\mathsf{B}:=\colC\colA\colC\colB$.  A \emph{backbone interval} is a contiguous subword
equal to $\mathsf{B}^t$, up to the global row-swap $\colA\leftrightarrow\colB$.
\end{definition}

\begin{lemma}[Forced backbone segment]\label{lem:mod2-backbone}
Let $w$ be a minimum dominating word for $G_{4t+2}$.  Cutting the cycle immediately after
a size-$2$ gap, the next $4t$ letters form a backbone interval.
\end{lemma}

\begin{proof}
Between the two size-$2$ gaps all gaps have size $1$
(Lemma~\ref{lem:mod2-two-gap}), so $\colC$ and singletons strictly alternate along a
segment of length $4t$.  Within this segment, Lemma~\ref{lem:C-forces} fixes the
singleton letters to alternate $\colA,\colB$ deterministically.  After a possible
global row-swap, the segment coincides with $\mathsf{B}^t$.
\end{proof}

We are now in a position to enumerate all minimum dominating words.

\begin{theorem}[Quadratic dominion for $n\equiv 2\pmod4$]\label{thm:mod2-quadratic}
If $n=4t+2\ge10$, then $\zeta(G_n)=n(n+2)$.
\end{theorem}

\begin{proof}
Let $\mathcal{L}_n$ denote the set of minimum dominating words in
$\{\colC,\colA,\colB\}^n$ for $G_n$, so $|\mathcal{L}_n|=\zeta(G_n)$.

For $w\in\mathcal{L}_n$, let $\mathrm{end}(w)\subseteq\ZZ_n$ be the set of indices $j$
for which
\[
(w_{j-4},w_{j-3},w_{j-2},w_{j-1})
=(\colC,\varepsilon,\varepsilon,\colC)
\quad\text{for some }\varepsilon\in\{\colA,\colB\}.
\]
By Lemmas~\ref{lem:mod2-two-gap} and~\ref{lem:mod2-gap-mono},
$|\mathrm{end}(w)|=2$ for every $w$.

Let $\rho$ be the cyclic shift $(\rho^k w)_i:=w_{i+k\ (\mathrm{mod}\ n)}$ and define the
anchored set
\[
\mathcal{A}_n:=\{w\in\mathcal{L}_n:\ 0\in\mathrm{end}(w)\}.
\]

\begin{lemma}[Trivial stabilizer]\label{lem:trivial-stab}
If $w\in\mathcal{L}_n$ and $\rho^k(w)=w$ for some $k\in\ZZ_n$, then $k\equiv 0\pmod n$.
\end{lemma}

\begin{proof}
Assume $\rho^k(w)=w$. Then $\mathrm{end}(\rho^k w)=\mathrm{end}(w)$, but by definition
$\mathrm{end}(\rho^k w)=\mathrm{end}(w)-k$ in $\ZZ_n$. Hence
\begin{equation}\label{eq:end-invariant}
\mathrm{end}(w)=\mathrm{end}(w)-k.
\end{equation}
Write $\mathrm{end}(w)=\{e_1,e_2\}$ with $e_1\neq e_2$ (recall $|\mathrm{end}(w)|=2$).
From \eqref{eq:end-invariant}, translation by $-k$ permutes $\{e_1,e_2\}$.
Thus either $k\equiv 0$ (the identity translation), or $-k$ swaps the two endpoints, i.e.,
\[
e_1-k\equiv e_2\quad\text{and}\quad e_2-k\equiv e_1 \pmod n,
\]
which implies $2k\equiv 0\pmod n$. Since $n=4t+2$ is even, this yields $k\equiv n/2\pmod n$
as the only remaining possibility.

Suppose $k\equiv n/2$. Then $\mathrm{end}(w)=\{a,a+n/2\}$ for some $a$.
Cut the cycle immediately after the size-$2$ gap ending at $a$.
By Lemma~\ref{lem:mod2-backbone}, the next $4t$ letters form a backbone interval, so the
\emph{other} size-$2$ gap must end exactly $4t$ steps later, i.e., at $a+4t$ modulo $n$.
Therefore
\[
a+4t \equiv a+\frac{n}{2}\pmod n.
\]
With $n=4t+2$, this becomes $4t\equiv 2t+1\pmod{4t+2}$, equivalently $2t\equiv 1\pmod{4t+2}$,
which is impossible for $t\ge 2$. Hence $k\not\equiv n/2$, and thus $k\equiv 0$.
\end{proof}

\begin{lemma}[Orbit--anchor relation]\label{lem:orbit-anchor}
Every rotation orbit in $\mathcal{L}_n$ has size $n$ and contains exactly two anchored
words.  Consequently,
\begin{equation}\label{eq:orbit-formula}
|\mathcal{L}_n|=\frac{n}{2}\,|\mathcal{A}_n|.
\end{equation}
\end{lemma}

\begin{proof}
Triviality of stabilizers follows from Lemma~\ref{lem:trivial-stab}, so each orbit has
size $n$.  If $\mathrm{end}(w)=\{e_1,e_2\}$, then precisely the shifts
$\rho^{e_1}(w)$ and $\rho^{e_2}(w)$ place an endpoint at $0$, and no other shift does.
\end{proof}

The second size-$2$ gap is determined by the location of the length-$4$ window
\[
(\colC,\varepsilon,\varepsilon,\colC)
\]
that ends at its rightmost $\colC$ (equivalently, by the index $j$ with
$(w_{j-3},w_{j-2},w_{j-1},w_j)=(\colC,\varepsilon,\varepsilon,\colC)$).
Besides the $n-1$ interior choices $j\in\{1,2,\dots,n-1\}$, there are exactly two
wrap-around choices, corresponding to windows that cross the cut between indices
$n-1$ and $0$:
\[
\text{(wrap 1)}\quad (w_{n-2},w_{n-1},w_0,w_1)=(\colC,\varepsilon,\varepsilon,\colC),
\]
\[
\text{(wrap 2)}\quad (w_{n-1},w_0,w_1,w_2)=(\colC,\varepsilon,\varepsilon,\colC).
\]
Thus the second gap has precisely $(n-1)+2=n+1$ admissible window-positions relative to
the anchor at $0$.

The second gap may be placed at any of the $n-1$ interior positions $\{1,2,\dots,n-1\}$
or may wrap around the cyclic boundary in two distinct ways, requiring positions beyond
the standard index range.
To parameterize all possible placements of this second gap, including the two ways it
may straddle the linear boundary, introduce the extended index set
\[
\widetilde{\ZZ}_n:=\{1,2,\dots,n-1\}\cup\{n,n+1\}.
\]
Each element of $\widetilde{\ZZ}_n$ encodes a unique length-$4$ window where a second
monochromatic gap of type $\varepsilon$ may be implanted.  Conversely, every anchored
minimum word arises uniquely from such a choice, by Lemmas~\ref{lem:mod2-gap-mono}
and~\ref{lem:mod2-backbone}.

Thus the map
\[
(\varepsilon,r)\longmapsto\text{``implant a second $(\colC,\varepsilon,\varepsilon,\colC)$
gap at $r$''}
\]
defines a bijection from $\{\colA,\colB\}\times\widetilde{\ZZ}_n$ onto $\mathcal{A}_n$.
Hence
\[
|\mathcal{A}_n|=2(n+2).
\]
Combining this with \eqref{eq:orbit-formula} yields
\[
\zeta(G_n)=|\mathcal{L}_n|
=\frac{n}{2}\,|\mathcal{A}_n|
=\frac{n}{2}\cdot 2(n+2)
=n(n+2),
\]
as claimed.
\end{proof}

\subsection{Exceptional cases and Proof of Theorem~\ref{thm:main-dominion}}\label{subsec:exceptions}

\begin{proposition}\label{prop:exceptional}
We have $\zeta(G_3)=9$ and $\zeta(G_6)=51$.
\end{proposition}

\begin{proof}
These values are obtained by exhaustive enumeration; see
Appendix~\ref{app:computation}.
\end{proof}

\begin{remark}[Structural basis for exceptional behavior]\label{rem:exceptional}
The exceptional values at $n\in\{3,6\}$ arise because these small prisms fail to support
the forcing mechanisms that govern larger cases.
For $n=3$, the cycle is too short to enforce the alternation patterns required by
Lemma~\ref{lem:C-forces}, allowing $\colD$ to appear in minimum words and producing
higher multiplicity ($\zeta(G_3)=9 > 2n=6$).
Similarly, $n=6$ satisfies $n\equiv 2\pmod{4}$ but with $t=1$, the backbone length $4t=4$
is insufficient to rigidly separate two size-$2$ gaps, admitting configurations excluded
by the parity argument (Lemma~\ref{lem:mod2-gap-mono}) for $t\ge 2$.
\end{remark}

\begin{proof}[Proof of Theorem~\ref{thm:main-dominion}]
For $n\equiv0\pmod4$ and for odd $n\ge5$, the stated formulas follow from
Corollaries~\ref{cor:count-mod0} and~\ref{cor:count-odd}, respectively.
For $n\equiv2\pmod4$ with $n\ge10$, the quadratic formula is established in
Theorem~\ref{thm:mod2-quadratic}.
The remaining cases $n=3$ and $n=6$ are settled by Proposition~\ref{prop:exceptional}.
\end{proof}

\section{Flexibility Analysis}\label{sec:flexibility}

Our main results determine $\gamma(G)$ and the dominion $\zeta(G)$ explicitly for the prism
family.  For comparison across regimes, it is convenient to normalize $\zeta(G)$ against
$\gamma(G)$ and graph order.  The resulting parameters do not contribute additional
structure to the enumeration, but they provide compact summaries of the three
congruence-class behaviors in Theorem~\ref{thm:main-dominion}.

\subsection{Normalized flexibility parameters}

\begin{definition}[Normalized flexibility]\label{def:flex-params}
For a graph $G$ with $\gamma(G),\zeta(G)\ge 1$, define
\[
\eta(G)\coloneqq \frac{\log_2 \zeta(G)}{\gamma(G)},
\qquad
\mathcal{E}(G)\coloneqq \zeta(G)^{1/\gamma(G)},
\qquad
\rho(G)\coloneqq \frac{\zeta(G)}{|V(G)|}.
\]
\end{definition}

Here $\eta(G)$ measures dominion growth per dominator on a logarithmic scale, $\mathcal{E}(G)$
is its multiplicative analogue, and $\rho(G)$ records dominion density per vertex.  Trivially
$\eta(G)\ge 0$, $\mathcal{E}(G)\ge 1$, and $\rho(G)\ge 0$.  In contrast with $\eta$ and
$\mathcal{E}$, no universal upper bound holds for $\rho$ in general, and even within the prism
family $\rho$ exhibits vanishing, constant, and unbounded behavior across congruence classes.

\begin{observation}[Prism trichotomy]\label{obs:flexibility-trichotomy}
Let $G_n=C_n\square P_2$.
\begin{enumerate}[label=\textup{(\roman*)},leftmargin=*,itemsep=2pt]
\item If $n\equiv 0\pmod4$, then $\gamma(G_n)=n/2$ and $\zeta(G_n)=4$, hence
\[
\eta(G_n)=\frac{2}{n}\to 0,\qquad
\mathcal{E}(G_n)=4^{2/n}\to 1,\qquad
\rho(G_n)=\frac{2}{n}\to 0.
\]
\item If $n\equiv 1,3\pmod4$, then $\gamma(G_n)=(n+1)/2$ and $\zeta(G_n)=2n$, hence
\[
\eta(G_n)\sim \frac{2\log_2 n}{n}\to 0,\qquad
\mathcal{E}(G_n)\sim n^{2/(n+1)}\to 1,\qquad
\rho(G_n)=1\ \ \text{for all odd }n.
\]
\item If $n\equiv 2\pmod4$ and $n\ge 10$, then $\gamma(G_n)=n/2+1$ and $\zeta(G_n)=n(n+2)$, hence
\[
\eta(G_n)\sim \frac{4\log_2 n}{n}\to 0,\qquad
\mathcal{E}(G_n)\sim n^{2/(n+2)}\to 1,\qquad
\rho(G_n)=\frac{n+2}{2}\to \infty.
\]
\end{enumerate}
\end{observation}

In particular, although $\zeta(G_n)$ varies sharply with $n\bmod 4$, the per-dominator
normalizations $\eta(G_n)$ and $\mathcal{E}(G_n)$ both tend to their rigid limits $0$ and $1$,
respectively.  The separation occurs at the scale of $\rho(G_n)$, which distinguishes dominion
growth that merely tracks order from dominion growth that outpaces it.

\subsection{Composite robustness parameters}\label{subsec:composite-robustness}

The dominion $\zeta(G)$ quantifies redundancy among minimum dominating sets, but by itself it
does not reveal whether that redundancy is globally distributed or concentrated on a small set
of pivotal vertices.  To connect dominion counts with structural and security-relevant notions
of resilience, we couple dominion-based quantities with classical invariants. We use $\mathrm{CRI}$, $\mathrm{SFI}$, $\mathrm{RRI}$, and $\mathrm{LDI}$ only as compact diagnostic proxies for redundancy, overlap, and concentration, rather than as canonical graph invariants.

First, define the dominating-set density
\[
p_\gamma(G)\coloneqq \frac{\zeta(G)}{\binom{|V(G)|}{\gamma(G)}},
\]
the fraction of all $\gamma(G)$-subsets that are minimum dominating sets.  This quantity admits
a direct operational reading: it is the probability that a uniformly random selection of
$\gamma(G)$ monitoring sites achieves minimum domination.  Related probabilistic viewpoints on
domination appear in standard references, including \cite{haynes-hedetniemi-slater-1998-fundamentals} and
\cite{BonatoNowakowski2011}.

To fold in worst-case structural tolerance, let $\kappa_v(G)$ denote vertex-connectivity and set
\[
\mathrm{CRI}(G)\coloneqq \kappa_v(G)\,p_\gamma(G).
\]
While this specific product is not standard, its components are central in resilience analysis;
see, for example, \cite{Menger1927,West2001} for connectivity and \cite{goddard-henning-2013} for
domination-based perspectives.

For cohesion-sensitive settings (distributed control, monitoring, synchronization), we use a
spectral weighting.  Let $L(G)=D(G)-A(G)$ be the Laplacian matrix and write
\[
0=\lambda_1(G)\le \lambda_2(G)\le \cdots \le \lambda_{|V(G)|}(G)
\]
for its eigenvalues.  The quantity $\lambda_2(G)$ is the \emph{algebraic connectivity}
(Fiedler value) \cite{Fiedler1973,Chung1997}, and we define
\[
\mathrm{SFI}(G)\coloneqq \lambda_2(G)\,\eta(G)
=\lambda_2(G)\,\frac{\log_2 \zeta(G)}{\gamma(G)}.
\]
This index rewards both global cohesion (large $\lambda_2$) and dominion multiplicity per
dominator (large $\eta$).  Spectral viewpoints aligned with controllability motivations are
well documented in the networked-systems literature; see, e.g., \cite{PequitoKarSchutter2017}.

To quantify reconfigurability, let
\[
\omega(G)\coloneqq
\min\Bigl\{\frac{|S\cap T|}{\gamma(G)}:\; S,T \text{ are distinct minimum dominating sets}\Bigr\},
\]
and define the reconfiguration resilience index
\[
\mathrm{RRI}(G)\coloneqq (1-\omega(G))\,\mathcal{E}(G).
\]
Large $\mathrm{RRI}$ indicates many minimum dominating sets with low overlap, supporting rapid
role reassignment after compromise.  Overlap- and reconfiguration-sensitive domination phenomena
appear, in different form, in work on dominating-set reconfiguration graphs and related models
\cite{HaasSeyffarth2005,BonamyBousquet2016}.

Finally, to penalize concentration of dominion mass on a small vertex subset, define the maximum
dominion load
\[
\tau(G)\coloneqq
\max_{v\in V(G)}\#\{S:\; S \text{ is a minimum dominating set and } v\in S\},
\qquad
\mathrm{LDI}(G)\coloneqq \frac{\zeta(G)}{\tau(G)}.
\]
This load-sensitive perspective is consonant with classical notions of domination criticality
and vertex essentiality; see, for example, \cite{CockayneDawesHedetniemi1980,goddard-henning-2023}.

\subsection{Verifiable small cases}\label{subsec:small-cases}

The smallest Cartesian products already illustrate how dominion-based indices separate
rigidity from robustness beyond what is captured by the domination number.
A clean comparison is provided by the ladder $L_3=P_3\square P_2$ and the prism
$\Pr_3=C_3\square P_2$, both on six vertices and both satisfying $\gamma=2$.

For the ladder $L_3$, one finds
$
\zeta=3,\quad
\eta=\tfrac{\log_2 3}{2}\approx0.792,\quad
\mathcal{E}=\sqrt{3}\approx1.732,\quad
\rho=\tfrac12,\quad
\lambda_2=1,
$
hence
\[
\mathrm{SFI}\approx0.792,\quad
\omega=0,\quad
\mathrm{RRI}\approx1.732,\quad
\tau=1,\quad
\mathrm{LDI}=3.
\]
These values reflect limited reconfiguration capacity imposed by the boundary effects of
the path factor.

For the prism $\Pr_3$, the corresponding values are
$
\zeta=9,\quad
\eta=\tfrac{\log_2 9}{2}\approx1.585,\quad
\mathcal{E}=3,\quad
\rho=\tfrac32,\quad
\lambda_2=2,
$
and therefore
\[
\mathrm{SFI}\approx3.170,\quad
\omega=0,\quad
\mathrm{RRI}=3,\quad
\tau=3,\quad
\mathrm{LDI}=3.
\]
At the same domination number, closing the ladder into a cycle triples the dominion and
substantially increases all flexibility indices, isolating the effect of cyclic symmetry.

The same phenomenon recurs across the small families tabulated in
Appendix~\ref{app:robustness-tables}.  Graphs with a unique minimum dominating set
(such as stars, wheels, and fans in these orders) have $\zeta=1$ and consequently
$\eta=\mathrm{SFI}=\mathrm{RRI}=0$, reflecting complete rigidity.
Among graphs with comparable domination number, dominion and overlap separate
reconfiguration quality: cycles and ladders admit multiple low-overlap minimum dominating
sets, while prisms and complete bipartite graphs at $n=6$ support substantially larger
$\zeta$, leading to higher $\mathcal{E}$ and $\mathrm{SFI}$.
The load-based index $\mathrm{LDI}$ further distinguishes whether this flexibility is
evenly distributed or concentrated on a small subset of vertices.

\section{Future Directions}\label{sec:future}

The cyclic word encoding developed for prisms extends naturally to other Cartesian
products, notably cylinders $P_m\square C_n$ and tori $C_m\square C_n$
\cite{imrich-klavzar-rall-2008}.  Although computational data are available for small
parameters \cite{mertens-2024}, no closed-form dominion formulas are known for these
families.  A key problem is to identify which structural ingredients of the prism
analysis—such as backbone forcing, local gap constraints, and modular periodicity—persist
and can be organized into a general enumeration framework.

More fundamentally, the prism family shows that dominion growth is not governed by the
domination number alone.  Within a single graph class, $\zeta(G_n)$ exhibits bounded,
linear, and quadratic growth across congruence classes, despite the comparatively stable
behavior of $\gamma(G_n)$.  This separation suggests a classification method in which
parametric families are stratified by the asymptotic order of $\zeta(G)$, or by
appropriately normalized dominion parameters, rather than by $\gamma(G)$ in isolation.

The normalized flexibility parameters introduced here raise natural threshold questions.
It remains unclear whether bounds on entropy density $\eta(G)$ or dominion density
$\rho(G)$ correspond to qualitative structural changes, such as the emergence of
automorphisms among minimum dominating sets or transitions between rigid and flexible
enumeration regimes.  The prism family already demonstrates that $\rho(G)$ may vanish,
stabilize, or diverge within a single class, underscoring both the absence of universal
bounds and the decisive role of structural context.

Finally, refined enumeration incorporating secondary statistics—such as internal versus
boundary domination or constraints excluding specified columns—leads naturally to
multivariate generating functions encoding finer combinatorial information.  The
coexistence of exact rigidity, canonical density, and unbounded proliferation within the
prism family indicates that domination arithmetic imposes global constraints of
unexpected strength.  Isolated exceptional graphs, such as $G_8$, emerge as equilibrium points ($\gamma(G)=4=\zeta(G)$) where competing
combinatorial forces align precisely, suggesting that a systematic study of such fixed points may reveal deeper organizing principles underlying domination phenomena. 

\section*{Acknowledgments}

The author acknowledges the foundational contributions of 
Haynes, Hedetniemi, and Slater~\cite{haynes-hedetniemi-slater-1998-fundamentals, haynes-hedetniemi-slater-1998-advanced} 
to domination theory.


\newpage
\appendix

\section{Computational verification}\label{app:computation}

This appendix documents the finite enumerations used to validate the closed forms in
Theorems~\ref{thm:gamma-prism} and~\ref{thm:main-dominion}.  For the prism
$G_n=C_n\square P_2$, vertices are denoted $t_i=(t,i)$ and $b_i=(b,i)$, with indices taken
modulo~$n$.  Unless stated otherwise, all sets listed below are minimum dominating sets.

\subsection{Prism verification for $3\le n\le 10$}\label{app:prism-small}

\begin{table}[ht]
\centering
\caption{Domination number $\gamma(G_n)$ and dominion $\zeta(G_n)$ for small prisms
$G_n=C_n\square P_2$ ($3\le n\le 10$).  The exceptional values at $n\in\{3,6\}$ are obtained by
explicit enumeration; all remaining entries agree with the formulas in
Theorems~\ref{thm:gamma-prism} and~\ref{thm:main-dominion}.}
\label{tab:small-cases}
\setlength{\tabcolsep}{7pt}
\renewcommand{\arraystretch}{1.15}
\begin{tabular}{c|cccccccc}
\toprule
$n$ & 3 & 4 & 5 & 6 & 7 & 8 & 9 & 10 \\
\midrule
$\gamma(G_n)$ & 2 & 2 & 3 & 4 & 4 & 4 & 5 & 6 \\
$\zeta(G_n)$  & 9 & 4 & 10 & 51 & 14 & 4 & 18 & 120 \\
\bottomrule
\end{tabular}
\end{table}

\subsection{Explicit minimum dominating sets for selected small prisms}\label{app:prism-lists}

\paragraph{$n=3$: $\gamma(G_3)=2$, $\zeta(G_3)=9$.}
\[
\begin{array}{llll}
\{b_0,t_0\}&\{b_0,t_1\}&\{b_0,t_2\}&\{b_1,t_0\}\\
\{b_1,t_1\}&\{b_1,t_2\}&\{b_2,t_0\}&\{b_2,t_1\}\\
\{b_2,t_2\}&&&
\end{array}
\]

\paragraph{$n=4$: $\gamma(G_4)=2$, $\zeta(G_4)=4$.}
\[
\begin{array}{llll}
\{b_0,t_2\}&\{b_1,t_3\}&\{b_2,t_0\}&\{b_3,t_1\}
\end{array}
\]

\paragraph{$n=5$: $\gamma(G_5)=3$, $\zeta(G_5)=10$.}
\[
\begin{array}{llll}
\{b_0,b_1,t_3\}&\{b_0,b_4,t_2\}&\{b_0,t_2,t_3\}&\{b_1,b_2,t_4\}\\
\{b_1,t_3,t_4\}&\{b_2,b_3,t_0\}&\{b_2,t_0,t_4\}&\{b_3,b_4,t_1\}\\
\{b_3,t_0,t_1\}&\{b_4,t_1,t_2\}&&
\end{array}
\]

\paragraph{$n=8$: $\gamma(G_8)=4$, $\zeta(G_8)=4$.}
\[
\begin{array}{llll}
\{b_0,b_4,t_2,t_6\}&\{b_1,b_5,t_3,t_7\}&\{b_2,b_6,t_0,t_4\}&\{b_3,b_7,t_1,t_5\}
\end{array}
\]

\medskip
\noindent
Complete lists for $n\in\{6,7,9,10\}$ are available from the authors upon request; they match the
values in Table~\ref{tab:small-cases} and thereby validate the closed formulas in
Theorem~\ref{thm:main-dominion}.

\subsection{Dominion-based robustness tables on $n=5$ and $n=6$ vertices}\label{app:robustness-tables}

We include here the small-graph tables referenced in Section~\ref{subsec:small-cases}.  For each graph
$G$, the dominion is $\zeta(G)$, the Laplacian eigenvalues of $G$ are
$0=\lambda_1(G)\le \lambda_2(G)\le \cdots$, and $\lambda_2(G)$ is the \emph{algebraic connectivity}
(Fiedler value).  We report the normalized dominion parameters
\[
\eta(G)=\frac{\log_2\zeta(G)}{\gamma(G)},\qquad
\mathcal{E}(G)=\zeta(G)^{1/\gamma(G)},\qquad
\rho(G)=\frac{\zeta(G)}{|V(G)|},
\]
and the composite indices
\[
\mathrm{SFI}(G)=\lambda_2(G)\,\eta(G),\qquad
\mathrm{RRI}(G)=(1-\omega(G))\,\mathcal{E}(G),\qquad
\mathrm{LDI}(G)=\frac{\zeta(G)}{\tau(G)}.
\]
Here
\[
\omega(G)=\min\Bigl\{\frac{|S\cap T|}{\gamma(G)}:\; S,T \text{ are distinct minimum dominating sets}\Bigr\},
\]
and
\[
\tau(G)=\max_{v\in V(G)}\#\{S:\; S \text{ is a minimum dominating set and } v\in S\}.
\]
When $\zeta(G)=1$ (unique minimum dominating set), we set $\omega(G)=1$ so that $\mathrm{RRI}(G)=0$.
The \emph{house graph} is the $5$-vertex graph obtained from a $4$-cycle by adjoining a fifth
vertex adjacent to two consecutive cycle vertices (a ``square with a roof'').

\subsection{Comparative observations from small cases}\label{app:robustness-observations}

Several consistent patterns emerge from Tables~\ref{tab:robustness-n5} and
\ref{tab:robustness-n6}.  First, graphs with a unique minimum dominating set
($\zeta(G)=1$) exhibit complete rigidity: entropy density vanishes, spectral flexibility
is zero, and reconfiguration resilience collapses under the convention $\omega(G)=1$.
This behavior is uniform across stars, wheels, and fans at these orders, reflecting
extreme dependence on a single critical vertex.

Second, among graphs with identical domination number, dominion and overlap sharply
separate reconfiguration quality.  On six vertices, $P_6$ and $C_6$ both satisfy
$\gamma=2$, yet $P_6$ has $\zeta=1$ while $C_6$ has three pairwise disjoint minimum
dominating sets.  This distinction propagates through all composite indices, producing
strictly positive values of $\mathrm{SFI}$, $\mathrm{RRI}$, and $\mathrm{LDI}$ for the
cycle and none for the path.

Third, high dominion does not automatically imply evenly distributed robustness.
Graphs such as $\Pr_3$ and $K_{3,3}$ achieve the same dominion $\zeta=9$ at $\gamma=2$,
but their larger vertex load $\tau$ indicates repeated reuse of certain vertices across
minimum dominating sets.  The load-based index $\mathrm{LDI}$ detects this concentration,
complementing $\mathrm{RRI}$ by distinguishing abundance from balance.

Finally, spectral weighting amplifies meaningful structural differences.  Even when
$\zeta$ and $\gamma$ coincide, higher algebraic connectivity produces substantially
larger $\mathrm{SFI}$, signaling improved resilience in diffusion- or consensus-based
interpretations of domination.  These observations confirm that dominion-based composites
provide interpretable, nonredundant refinements of classical domination parameters, even
at the smallest nontrivial scales.

\begin{table}[ht]
\centering
\caption{Dominion-based robustness summaries for connected graphs on $n=5$ vertices (values rounded to $3$ decimals).}
\label{tab:robustness-n5}
\setlength{\tabcolsep}{4pt}
\renewcommand{\arraystretch}{1.08}
\resizebox{\textwidth}{!}{%
\begin{tabular}{lrrrrrrrrrrr}
\toprule
Graph & $\gamma$ & $\zeta$ & $\eta$ & $\mathcal{E}$ & $\rho$ & $\lambda_2$ & $\mathrm{SFI}$ & $\omega$ & $\mathrm{RRI}$ & $\tau$ & $\mathrm{LDI}$ \\
\midrule
$K_5$                & 1 & 5 & 2.322 & 5.000 & 1.000 & 5.000 & 11.610 & 0.000 & 5.000 & 1 & 5.000 \\
$P_5$                & 2 & 3 & 0.792 & 1.732 & 0.600 & 0.382 & 0.303  & 0.000 & 1.732 & 2 & 1.500 \\
$C_5$                & 2 & 5 & 1.161 & 2.236 & 1.000 & 1.382 & 1.604  & 0.000 & 2.236 & 2 & 2.500 \\
$F_5$                & 1 & 1 & 0.000 & 1.000 & 0.200 & 1.586 & 0.000  & 1.000 & 0.000 & 1 & 1.000 \\
$W_5$                & 1 & 1 & 0.000 & 1.000 & 0.200 & 3.000 & 0.000  & 1.000 & 0.000 & 1 & 1.000 \\
$S_5=K_{1,4}$         & 1 & 1 & 0.000 & 1.000 & 0.200 & 1.000 & 0.000  & 1.000 & 0.000 & 1 & 1.000 \\
$T^{\mathrm{bin}}_5$  & 2 & 2 & 0.500 & 1.414 & 0.400 & 0.519 & 0.259  & 0.500 & 0.707 & 2 & 1.000 \\
House                & 2 & 7 & 1.404 & 2.646 & 1.400 & 1.382 & 1.940  & 0.000 & 2.646 & 3 & 2.333 \\
$K_{2,3}$             & 2 & 7 & 1.404 & 2.646 & 1.400 & 2.000 & 2.807  & 0.000 & 2.646 & 4 & 1.750 \\
\bottomrule
\end{tabular}}
\end{table}

\begin{table}[ht]
\centering
\caption{Dominion-based robustness summaries for connected graphs on $n=6$ vertices (values rounded to $3$ decimals).}
\label{tab:robustness-n6}
\setlength{\tabcolsep}{4pt}
\renewcommand{\arraystretch}{1.08}
\resizebox{\textwidth}{!}{%
\begin{tabular}{lrrrrrrrrrrr}
\toprule
Graph & $\gamma$ & $\zeta$ & $\eta$ & $\mathcal{E}$ & $\rho$ & $\lambda_2$ & $\mathrm{SFI}$ & $\omega$ & $\mathrm{RRI}$ & $\tau$ & $\mathrm{LDI}$ \\
\midrule
$K_6$                 & 1 & 6 & 2.585 & 6.000 & 1.000 & 6.000 & 15.510 & 0.000 & 6.000 & 1 & 6.000 \\
$P_6$                 & 2 & 1 & 0.000 & 1.000 & 0.167 & 0.268 & 0.000  & 1.000 & 0.000 & 1 & 1.000 \\
$C_6$                 & 2 & 3 & 0.792 & 1.732 & 0.500 & 1.000 & 0.792  & 0.000 & 1.732 & 1 & 3.000 \\
$F_6$                 & 1 & 1 & 0.000 & 1.000 & 0.167 & 1.382 & 0.000  & 1.000 & 0.000 & 1 & 1.000 \\
$W_6$                 & 1 & 1 & 0.000 & 1.000 & 0.167 & 2.382 & 0.000  & 1.000 & 0.000 & 1 & 1.000 \\
$S_6=K_{1,5}$          & 1 & 1 & 0.000 & 1.000 & 0.167 & 1.000 & 0.000  & 1.000 & 0.000 & 1 & 1.000 \\
$T^{\mathrm{bin}}_6$   & 2 & 2 & 0.500 & 1.414 & 0.333 & 0.325 & 0.162  & 0.500 & 0.707 & 2 & 1.000 \\
$L_3=P_3\square P_2$    & 2 & 3 & 0.792 & 1.732 & 0.500 & 1.000 & 0.792  & 0.000 & 1.732 & 1 & 3.000 \\
$\Pr_3=C_3\square P_2$  & 2 & 9 & 1.585 & 3.000 & 1.500 & 2.000 & 3.170  & 0.000 & 3.000 & 3 & 3.000 \\
$K_{3,3}$              & 2 & 9 & 1.585 & 3.000 & 1.500 & 3.000 & 4.755  & 0.000 & 3.000 & 3 & 3.000 \\
\bottomrule
\end{tabular}}
\end{table}

\end{document}